\documentclass[11pt,a4paper]{article}
\usepackage[latin1]{inputenc}
\usepackage{amsmath}
\usepackage{amsthm}
\usepackage{amsfonts}
\usepackage{amsfonts,amsthm,latexsym,amsmath,amssymb,amscd,epsfig,psfrag,enumerate}
\usepackage{graphics,graphicx, bezier, float, color, hyperref}
\usepackage{amssymb,url}
\usepackage{multienum}
\usepackage[table]{xcolor}
\usepackage{multicol,multirow}
\usepackage{graphicx}
\usepackage{fancyvrb}
\usepackage{parskip}
\usepackage[toc,page]{appendix}
\usepackage[top=2.7cm, bottom=2.7cm, left=1.5cm, right=1.5cm]{geometry}
\usepackage{xcolor}

\usepackage{blkarray}
\newtheorem{theorem}{Theorem}[section]
\newtheorem{lemma}{Lemma}[section]

\usepackage[none]{hyphenat}[section]
\newtheorem{definition}{Definition}[section]

\numberwithin{equation}{section}
\numberwithin{table}{section}
\numberwithin{figure}{section}

\title{Perrin numbers that are palindromic concatenations of two repdigits}
\author{Herbert Batte$^{1} $ and Prosper Kaggwa$^{2,*} $}
\date{}

\begin{document}
	\maketitle
\abstract{ Let $ \{P_n\}_{n\geq 0} $ be the sequence of Perrin numbers defined by $P_0=3$, $P_1=0$,$P_2=2$ and $P_{n+3}=P_{n+1}+P_{n}$ for all $n \geq 0$. In this paper, we determine all Perrin numbers that are palindromic concatenations of two repdigits. } 

 

\thanks{$ ^{*} $ Corresponding author}

\section{Introduction}
\subsection{Background}
Consider the sequence of Perrin numbers $\{P_n\}_{n\geq 0}$, defined with the initial values $P_0 = 3$, $P_1 = 0$, $P_2 = 2$, and following the recurrence relation $P_{n+3} = P_{n+1} + P_n$ for all $n \geq 0$. The beginning of this sequence is:
$$
3,\, 0,\, 2,\, 3,\, 2,\, 5,\, 5,\, 7,\, 10,\, 12,\, 17,\, 22,\, 29,\, 39,\, 51, \ldots.
$$
A \textit{repdigit} in base 10 is a positive integer $N$ composed of repeated occurrences of a single digit. Specifically, $N$ can be expressed as:
\[
N = \underbrace{\overline{d \cdots d}}_{\ell \text{ times}} = d \left( \frac{10^\ell - 1}{9} \right),
\]
where $d$ and $\ell$ are positive integers, with $0 \leq d \leq 9$ and $\ell \geq 1$. 

This work builds upon prior research on the Diophantine properties of sequences generated by recurrence relations, specifically examining terms that can be represented within the sequence or as combinations of sequence terms. Despite the extensive work by Luca and Banks \cite{banks}, their results provided limited insights into the frequency of such terms within sequences. The problem of identifying Fibonacci numbers formed by two repdigits was investigated in \cite{ala}, revealing the largest such number as \(F_{14} = 377\).

Further investigations into the connection between linear recurrence sequences and repdigits have been carried out. For example, \cite{gar} determined all repdigits formed by adding two Padovan numbers. Ddamulira expanded this by exploring Padovan numbers that are concatenations of two distinct repdigits, identifying \(P_{21} = 200\) as the largest such number in \cite{dda2}.

Additional significant contributions to this area include works by Bedná\v rik and Trojovská \cite{bed}, Boussayoud et al. \cite{bou}, and Bravo and Luca \cite{bravo}. Among these, \cite{er} identified the only Perrin numbers that are concatenations of two distinct repdigits as \( P_n \in \{10, 12, 17, 22, 29, 39, 51, 68, 90, 119, 227, 644\} \).

This result from \cite{er} was derived using different methodologies. A logical extension of this research would be to identify Perrin numbers that are palindromic concatenations of two repdigits. A number is considered a \textit{palindrome} if it reads the same forwards and backwards. As a preliminary step, we focus on a more specific Diophantine equation:
\begin{align}\label{eq1.1l}
	P_n = \overline{\underbrace{d_1 \ldots d_1}_{\ell \text{ times}}\underbrace{d_2 \ldots d_2}_{m \text{ times}}\underbrace{d_1 \ldots d_1}_{\ell \text{ times}}},
\end{align} 
where $d_1, d_2 \in \{0, 1, 2, \ldots, 9\}$ and $d_1 > 0$. Similar investigations, such as in \cite{chal}, have proven that 151 and 616 are the only Padovan numbers that are palindromic concatenations of two distinct repdigits.

Here, we present the following result.
\subsection{Main Results}\label{sec:1.2l}
\begin{theorem}\label{thm1.1l} 
	22 is the only Perrin number that is a palindromic concatenation of two repdigits.
\end{theorem}

\section{Methods}
\subsection{Preliminaries}
For all $n \ge 0$, the Binet formula for the Perrin numbers tells us that the $n^\text{th}$ Perrin number is given by
\begin{equation}
P_n = \alpha^n + \beta^n + \gamma^n, \label{eq:binet}
\end{equation}
where
\begin{equation*}
\begin{aligned}
\alpha &= \frac{r_1 + r_2}{6},\qquad & \beta &= \frac{-(r_1 + r_2) + i(\sqrt{3}(r_1 - r_2))}{6} \qquad\text{and}\qquad & \gamma &= \overline{\beta},
\end{aligned}
\end{equation*}
with
\begin{align*}
r_1 &= \frac{31 + \sqrt{69}}{8} \qquad\text{and}\qquad r_2 = \frac{31 - \sqrt{69}}{8}.
\end{align*}
Moreover, the characteristic equation for the Perrin sequence is given by $ \phi(x) = x^3 - x - 1 = 0 $
with zeros $\alpha$, $\beta$ and $\gamma$ as defined above. Numerically, the following estimates hold for the quantities $\{\alpha, \beta, \gamma\}$:
\begin{equation*}
1.32 < \alpha < 1.33, 
\end{equation*}
\begin{equation*}
0.86 < |\beta| = |\gamma| = \alpha^ {- \frac{1}{2}} < 0.87. 
\end{equation*}
Here we can see that the complex conjugate roots $\beta$ and $\gamma$ have minor contributions to the right--hand side of  \eqref{eq:binet}. More specifically, if
\[
e(n) := P_n - \alpha^n = \beta^n + \gamma^n,
\]
then 
\begin{equation*}
|e(n)| < \frac{3}{\alpha^{n/2}} \quad \text{for all } n \ge 0. \label{eq:error_bound}
\end{equation*}
Next, the following estimate also holds.
\begin{lemma}\label{lem1}
Let $n \ge 2$ be a positive integer, then 
\begin{equation*}
	\alpha^{n-2} \le P_n \le \alpha^{n+1}. \label{eq:lemma1}
\end{equation*}	
\end{lemma}
\begin{proof}
	The proof of Lemma \ref{lem1} follows from a simple inductive argument, and the fact that $\alpha^3=\alpha+1$, from the
	characteristic polynomial $\phi$.
\end{proof}
Additionally, we note that \eqref{eq1.1l} can be written as 
\begin{equation}\label{eq:pn_value}
	P_n = \frac{1}{9} \left( d_1 \cdot 10^{2\ell + m} - (d_1 - d_2) \cdot 10^{\ell + m} + (d_1 - d_2) \cdot 10^\ell - d_1 \right), 
\end{equation}
so that
	\begin{equation*}
		\alpha^{n+1} \ge P_n = \frac{1}{9} \left( d_1 \cdot 10^{2\ell + m} - (d_1 - d_2) \cdot 10^{\ell + m} + (d_1 - d_2) \cdot 10^\ell - d_1 \right) > 10^{2\ell + m},
	\end{equation*}
or $\alpha^{n+1} > 10^{2\ell + m - 1} $. Taking logs on both sides yields 
\begin{align}
	2\ell +m-2<n.
\end{align}
Let $K := \mathbb{Q}(\alpha, \beta)$ be the splitting field of the polynomial $\phi$ over $\mathbb{Q}$. Then $[K: \mathbb{Q}] = 6$ and $[\mathbb{Q}(\alpha) : \mathbb{Q}] = 3$. We note that the Galois group of $K/\mathbb{Q}$ is given by

\[
\mathcal{G} := \mathrm{Gal}(K/\mathbb{Q}) \cong \{(1), (\alpha\beta), (\alpha\gamma), (\beta\gamma), (\alpha\beta\gamma)\} \cong S_3.
\]

We therefore identify the automorphisms of $\mathcal{G}$ with the permutation group of the zeroes of $\phi$. We highlight the permutation $(\alpha\beta)$, corresponding to the automorphism $\sigma: \alpha \mapsto \beta, \beta \mapsto \alpha, \gamma \mapsto \gamma$, which we use later to obtain a contradiction on the size of the absolute value of a certain bound.
\subsection{Linear Forms in Logarithms}
We use three times Baker--type lower bounds for nonzero linear forms in three logarithms of algebraic numbers. There are many such bounds mentioned in the literature like that of Baker and W\"ustholz from \cite{BW} or Matveev from \cite{matl}. Before we can formulate such inequalities, we need the notion of height of an algebraic number recalled below.

\begin{definition}[Logartihmic height]
	Let $\gamma$ be an algebraic number of degree $d$ with minimal primitive polynomial over the integers, given by
	\[
	a_0 x^d + a_1 x^{d-1} + \cdots + a_d = a_0 \prod_{i=1}^d (x - \gamma^{(i)}),
	\]
	where the leading coefficient $a_0$ is positive. Then the logarithmic height of $\gamma$ is given by 
	\begin{equation*}
	h(\gamma) := \frac{1}{d} \left( \log a_0 + \sum_{i=1}^d \log \max \{ |\gamma^{(i)}|, 1 \} \right). \label{eq:log_height}
	\end{equation*}
\end{definition} 
In particular, if $\gamma$ is a rational number represented as $\gamma := p/q$ with coprime integers $p$ and $q \ge 1$, then 
\[
h(\gamma) = \log \max \{ |p|, q \}.
\]
The following properties of the logarithmic height function $h(\cdot)$ will be used in the rest of the paper without further reference:
\begin{align*}
	h(\gamma_1 \pm \gamma_2) &\le h(\gamma_1) + h(\gamma_2) + \log 2,\\
h(\gamma_1 \gamma_2^{\pm 1}) &\le h(\gamma_1) + h(\gamma_2),\\
h(\gamma^s) &= |s| h(\gamma) \quad \text{valid for } s \in \mathbb{Z}.	
\end{align*}

A linear form in logarithms is an expression
\begin{align}\label{eqn:LinearFormsInLogarithms}
\Lambda := b_1 \log \gamma_1 + \cdots + b_t \log \gamma_t ,
\end{align}
where for us $\gamma_1, \ldots, \gamma_t$ are positive real algebraic numbers and $b_1, \ldots, b_t$ are non--zero integers. We assume $\Lambda \neq 0$. We need lower bounds for $|\Lambda|$. We write $K := \mathbb{Q}(\gamma_1, \ldots, \gamma_t)$ and $D$ for the degree of $K$. We start with the general form due to Matveev in  \cite{matl}.

\begin{theorem}[Matveev, \cite{matl}]\label{log}
Put $\Gamma := \gamma_1^{b_1} \cdots \gamma_t^{b_t} - 1 = e^\Gamma - 1$. Assume $\Gamma \neq 0$. Then 
\begin{equation*}
\log |\Gamma| > -1.4 \cdot 30^{t+3} \cdot t^{4.5} \cdot D^2 (1 + \log D) (1 + \log B) A_1 \cdots A_t, \label{eq:matveev}
\end{equation*}
where $B \ge \max \{|b_1|, \ldots, |b_t|\}$ and $A_i \ge \max \{ D h(\gamma_i), |\log \gamma_i|, 0.16 \}$ for all $i = 1, \ldots, t$.
\end{theorem}

\subsection{Reduction Methods}

Typically, the estimates from Matveev's theorem are excessively large to be practical in computations. To refine these estimates, we employ a modified approach based on the Baker--Davenport reduction method. Our adaptations follow the method introduced by Dujella and Petho (\cite{duj}, Lemma 5). When considering a real number $r$, we use $\| r \|$ to represent the smallest distance between $r$ and any nearest integer which is formally written as $\min \{ |r - n| : n \in \mathbb{Z} \}$.
\begin{lemma}[Dujella \& Petho, \cite{duj}]\label{duj}
Let $\tau \neq 0$ and $A, B, \mu$ be real numbers with $A > 0$ and $B > 1$. Let $M > 1$ be a positive integer and suppose that $p/q$ is a convergent of the continued fraction expansion of $\tau$ with $q > 6M$. Let 
\[
\varepsilon := ||\mu q|| - M \|| \tau q \||.
\]
If $\varepsilon > 0$, then there is no solution of the inequality 
\begin{equation*}
0 <| m \tau - n + \mu |< A B^{-k},
\end{equation*}
in positive integers $m, n, k$ with 
\begin{equation*}
\frac{\log (Aq/\varepsilon)}{\log{B}} \ge k \quad \text{and} \quad m < M.
\end{equation*}
\end{lemma}
Finally, we present an analytical argument which is Lemma 7 in \cite{guz}.
\begin{lemma}[Lemma 7 in \cite{guz}]\label{reduction}
If $m \ge 1$, $T > (4m^2)^m$ and $T > Z/(\log{Z})^m$, then 
\begin{equation*}
Z < 2^mT( \log T)^m.
\end{equation*}
\end{lemma}
SageMath was used to perform all the computations in this work.

\section{Proof of Theorem 1.1}
\subsection{The Low Range $n \leq 700$}
Using a basic SageMath script, we investigated all possible solutions to the Diophantine equation \eqref{eq1.1l} for parameters $d_1, d_2 \in \{0, \ldots, 9\}$ with $d_1 > 0$. We restricted our search to $1 \leq \ell, m \leq n \leq 700$, where $\ell$ and $m$ are the lengths of the two blocks of repeated digits in $P_n$. We found only one solution, which is given in Theorem \ref{thm1.1l}. Henceforth, we assume $n > 700$.

\subsection{The Initial Bound on $n$}
We now proceed to examine \eqref{eq1.1l} in different ways. We first prove the following result.
\begin{lemma}\label{lbound}
Let $l$, $m$ and $n > 700$ be solutions to the Diophantine equation \eqref{eq1.1l}, then 
\begin{equation*}
l < 4.1 \cdot 10^{13} \log n. 
\end{equation*}
\end{lemma}
\begin{proof}
 We go back to \eqref{eq:pn_value} and rewrite it using \eqref{eq:binet} as
\begin{align*}
    \alpha^n + \beta^n + \gamma^n &=\frac{1}{9} \left( d_1 \cdot 10^{2\ell + m} - (d_1 - d_2) \cdot 10^{\ell + m} + (d_1 - d_2) \cdot 10^\ell - d_1 \right), \\
    9 (\alpha^n + \beta^n + \gamma^n) &= d_1 \cdot 10^{2\ell+m} - (d_1 - d_2) \cdot 10^{\ell+m} + (d_1 - d_2) \cdot 10^\ell - d_1, \\
    9 \alpha^n - d_1 \cdot 10^{2\ell+m} &= -9 e(n) - (d_1 - d_2) \cdot 10^{\ell+m} + (d_1 - d_2) \cdot 10^\ell - d_1 .\end{align*}
Therefore we have that \begin{align*}
    \left|9 \alpha^n - d_1 \cdot 10^{2\ell+m}\right| &= \left|-9 e(n) - (d_1 - d_2) \cdot 10^{\ell+m} + (d_1 - d_2) \cdot 10^\ell - d_1\right| \\
    &< 27 \alpha ^{-n/2}  + 27 \cdot 10^{\ell+m}, 
\end{align*}
so that $\left|9 \alpha^n - d_1 \cdot 10^{2\ell+m}\right| <28 \cdot 10^{\ell+m}$, for all $n>700$.
Dividing both sides by \( d_1 \cdot 10^{2\ell+m} \), we get
\begin{equation}
    \left|\frac{9}{d_1}\cdot \alpha^n \cdot 10^{-2\ell-m}-1\right| < 28 \cdot 10^{-\ell}.
    \label{eq:3.1}
\end{equation}
Let \begin{align*}
\Gamma = (9/d_1)\cdot \alpha^n \cdot 10^{-2\ell-m}-1.
\end{align*} Notice that \(\Gamma \neq 0\) otherwise we would have \begin{align*}
\alpha^n = \frac{d_1 \cdot 10^{2l+m}}{9}.
\end{align*}
If this is the case, then applying the automorphism \(\sigma\) on both sides of the above equation and taking absolute values we have that 
\[
   1< \left|\frac{10^{2l+m}\cdot d_1}{9}\right| = |\sigma(\alpha^n)| = |\beta^n| < 1,
\] 
which is false. Thus \(\Gamma \ne 0\). We use the field \( \mathbb{Q}(\alpha) \) with
\begin{align*}
& \lambda_1 := \frac{9}{d_1}, \quad \lambda_2 := \alpha, \quad \lambda_3 := 10, \quad b_1 := 1, \quad b_2 := n, \\
& b_3 := -2\ell-m, \quad B := n, \quad t := 3, \quad D := 3.
\end{align*}
Moreover,
\begin{align*}
    & h(\lambda_1) := h\left(\frac{9}{d_1}\right) \leq 2h(9) := 2\log 9 < 5, \quad h(\lambda_2) := h(\alpha)=\frac{\log\alpha}{3}, \quad h(\lambda_3) := h(10) := \log {10}, 
\end{align*}
   so we can take  $A_1 := 3 \cdot 5 = 15$, $A_2:= 3 (\log\alpha)/3 = \log \alpha$, $A_3 =3 \cdot \log{10} = 3\log{10}$. Now, by Theorem \ref{log}, we have
\begin{equation}
    \log \Gamma > -1.4 \cdot 30^6 \cdot 3^{4.5} \cdot 3^2 \cdot (1 + \log 3)(1 + \log n) \cdot 15 \cdot \log \alpha \cdot 3 \log 10 \\ > -9.3 \cdot 10^{13} \log n. \label{eq:3.2}
\end{equation}
Comparing \eqref{eq:3.1} and \eqref{eq:3.2}, we get \begin{equation*}l\log{10}-\log{28}<9.3\cdot10^{13}\log{n},
	\end{equation*}which leads to $l<4.1\cdot 10^{13}\log{n}$
	 for all $n>700$.
\end{proof}
Next, we prove the following.
\begin{lemma}\label{mbound}
Let \( l \), \( m \) and \(n>700\) be solutions to the Diophantine equation \eqref{eq1.1l} ,Then
\[
    m < 1.6 \cdot 10^{27}( \log n)^2.
\]
\end{lemma}
\begin{proof}
Again we go back to \eqref{eq:pn_value} and rewrite it with \eqref{eq:binet} as,
\begin{align*}
    \alpha^n + \beta^n + \gamma^n &= \frac{1}{9} \left( d_1 \cdot 10^{2\ell + m} - (d_1 - d_2) \cdot 10^{\ell + m} + (d_1 - d_2) \cdot 10^\ell - d_1 \right), \\
    9 (\alpha^n + \beta^n + \gamma^n) &= d_1 \cdot 10^{2\ell + m} - (d_1 - d_2) \cdot 10^{\ell + m} + (d_1 - d_2) \cdot 10^\ell - d_1 ,\\
9 \alpha^n - d_1 \cdot 10^{2\ell + m} + (d_1 - d_2) \cdot 10^{\ell + m} &= -9 e(n) + (d_1 - d_2) \cdot 10^\ell - d_1.
    \end{align*}
Taking absolute values both sides, we have
\begin{align*}
\left|9 \alpha^n - d_1 \cdot 10^{2\ell + m} + (d_1 - d_2) \cdot 10^{\ell + m}\right| &\leq \left|-9 e(n) + (d_1 - d_2) \cdot 10^\ell - d_1\right|\\
&\leq 27 \alpha^{-n/2} + 18 \cdot 10^l \\
&< 19 \cdot 10^l,
\end{align*}  
for all $n>700$. Now, dividing both sides by \((d_1\cdot10^l-(d_1-d_2))\cdot10^{l+m}\), we get \begin{equation}
\left|\frac{9}{(d_1\cdot10^l-(d_1-d_2))}\cdot\alpha^n\cdot10^{-l-m}-1\right|<\frac{19}{(d_1\cdot10^l-(d_1-d_2))}\cdot 10^{-m} <19\cdot10^{-m}.
\label{eq:3.3}
\end{equation}
Let \begin{equation*}
\Gamma_1 = \frac{9}{(d_1\cdot10^l-(d_1-d_2))}\cdot\alpha^n\cdot10^{-l-m}-1.
\end{equation*}
Clearly, $\Gamma_1 \neq 0$ otherwise we would have \begin{equation*}
\alpha^n = \frac{(d_1\cdot10^l - (d_1 - d_2))\cdot10^{l+m}}{9}.	
\end{equation*}
Applying the automorphism $\sigma$ on both sides of the above equation and taking absolute values we have that
\begin{equation*}
1<\left|\frac{9}{(d_1\cdot10^l-(d_1-d_2))}\right|=|\sigma(\alpha^n)|=|\beta^n| < 1, 	
\end{equation*} 
 which is false. We thus have that  $\Gamma_1 \neq 0$.
We can use the field \( \mathbb{Q}(\alpha) \) with:
\begin{align*}
    \lambda_1 &:= \frac{9}{(d_1\cdot10^l-(d_1-d_2))}, & \lambda_2 &:= \alpha, & \lambda_3 &:= 10, \\
    b_1 &:= 1, & b_2 &:= n, & b_3 &:= -l-m, \\
    B &:= n, & D &:= 3, & t &:= 3.
\end{align*}
Moreover,
\begin{align*}
     h(\lambda_1) &:= h\left(\frac{9}{(d_1\cdot10^l-(d_1-d_2))}\right) 
    \leq 2(h(9)+h(d_1)+h(10^l)+h(d_1)+h(d_2)+2\log2) \\
    &\leq 2(\log9+\log9+l\log10+\log9+\log9+2\log2) 
    \leq 8\log9+2l\log10+4\log2 \\
    &\leq 8\log9+4\log2+2(4.1\cdot 10^{13}\log n)\log10 < 1.9\cdot 10^{14}\log n ,
\end{align*}
so we can take $
 	 A_1 := 3 \cdot 1.9 \cdot 10^{14} \log n = 5.7 \cdot 10^{14} \log n$, and as before
    $A_2 = \log\alpha$  and $A_3 = 3\log10$ .

By Theorem \ref{eq:matveev},  
\begin{align}\label{3.4}
    \log \Gamma_1 &> -1.4 \cdot 30^6 \cdot 3^{4.5} \cdot 3^2 \cdot (1 + \log 3)(1 + \log n) \cdot 5.7\cdot 10^{14}\log{n}\cdot\log\alpha\cdot 3\log10 \nonumber \\
    &> -3.5 \cdot 10^{27} \cdot (\log n)^2 \end{align}
Comparing \eqref{eq:3.3} to \eqref{3.4}, we get
$m<1.6\cdot 10^{27}\left(\log{n}\right)^2$, for all  $n>700$.
 \end{proof}

Lastly in this sub--section, we prove the following
\begin{lemma}\label{nbound}
Let \( l \), \( m \) and \(n>700\) be solutions to the Diophantine equation \eqref{eq1.1l}. Then, 
\begin{equation*}
 \ell<4.6 \cdot 10^{15},\qquad m<2.0 \cdot 10^{31} \qquad\text{and}\qquad	n<2.8 \cdot 10^{48}.
\end{equation*}
\end{lemma}
\begin{proof}
Once more, we go back to \eqref{eq:pn_value} and rewrite it with \eqref{eq:binet} as
\begin{align*}
    9\alpha^n - (d_1\cdot10^{l+m} - (d_1-d_2)\cdot10^m + (d_1-d_2))\cdot 10^l &= -9e(n) - d_1 .
    \end{align*} 
Therefore, we have that \begin{align*}
	\left|9\alpha^n - (d_1\cdot10^{l+m} - (d_1-d_2)\cdot10^m + (d_1-d_2))\cdot 10^l\right| &= |-9e(n) - d_1|\leq 9e(n) + 9< 10, 
\end{align*}
where we used the fact that \(n>700\). Now, dividing both sides by \(9\alpha^n\), we get \begin{equation}
\left|{\frac{d_1\cdot10^{l+m} - (d_1-d_2)\cdot10^m + (d_1-d_2)}{9}}\cdot \alpha^{-n} \cdot10^l-1\right|<\frac{10}{9}\cdot \alpha^{-n}.
\label{eq:3.5}
\end{equation}
Let 
\begin{equation*}
	\Gamma_2 =\frac{d_1 \cdot 10^{l+m} - (d_1 - d_2) \cdot 10^m + (d_1 - d_2)}{9} \cdot \alpha^{-n} \cdot 10^l - 1. \end{equation*} 
 Clearly, \(\Gamma_2 \neq 0\) otherwise we would have, \begin{equation*}
 	\alpha^n = \frac{d_1 \cdot 10^{l+m} - (d_1 - d_2) \cdot 10^m + (d_1 - d_2)}{9}\cdot 10^l. 
 \end{equation*}
Now, applying the automorphism \(\sigma\) on both sides of the above equation as before and taking absolute values, we get
\[
1<\left| \left( \frac{d_1 \cdot 10^{l+m} - (d_1 - d_2) \cdot 10^m + (d_1 - d_2)}{9} \right) \cdot 10^l\right| = |\sigma(\alpha^n)| = |\beta^n| < 1,
\]
 which is false. Hence \(\Gamma_2 \neq 0\). We again use the field \( \mathbb{Q}(\alpha) \) with 
\begin{align*}
\Lambda_1 & := \frac{d_1 \cdot 10^{l+m} - (d_1 - d_2) \cdot 10^m + (d_1 - d_2)}{9}, & \lambda_2 & := \alpha, & \lambda_3 & := 10, \\
b_1 & := 1, & b_2 & := -n, & b_3 & := l, \\
D & := 3, & t & := 3, & A_1 & := Dh(\lambda_1).
\end{align*}
Notice that
\begin{align*}
h(\lambda_1) & = h\left(\frac{d_1 \cdot 10^{l+m} - (d_1 - d_2) \cdot 10^m + (d_1 - d_2)}{9}\right) \\
& \leq h(9) + (l+m)h(10) + h(d_1 - d_2) + mh(10) + h(d_1 - d_2) + 3 \log 2 \\
& < 7 \log 9 + (l+m) \log 10 + m \log 10 \\
& < 7 \log 9 + ((4.1 \cdot 10^{13} \log n + 1.6 \cdot 10^{27} (\log n)^2) \log 10 + 1.6 \cdot 10^{27} (\log n)^2 \log 10) \\
& < 6.0 \cdot 10^{27} (\log n)^2,
\end{align*}
so we can take $A_1  := 3 \cdot 6.0 \cdot 10^{27} (\log n)^2 := 1.8 \cdot 10^{28} (\log n)^2$, and as before  $ A_2 =  \log \alpha$ and $ A_3 = 3 \log 10$. By Theorem \ref{eq:matveev},
\begin{equation}
\log \Gamma_2 > -1.4 \cdot 30^6 \cdot 3^{4.5} \cdot 3^2 \cdot (1 + \log 3)(1 + \log n) \cdot 1.8 \cdot 10^{28} (\log n)^2 \cdot \log \alpha \cdot 3 \log 10. \label{3.6}
\end{equation}
Comparing \eqref{eq:3.5} and \eqref{3.6}, we get
\begin{equation*}
n \log\alpha-\log{10}<1.1\cdot 10^{41} (\log {n})^3,
\end{equation*}
which leads to \(n<3.9\cdot10^{41}(\log n)^3,\) for all \(n>700.\)

To proceed from here, let \(z=n, m=3, T=3.9\cdot 10^{41},\) then Lemma \ref{reduction} implies that \(n<2^3\cdot 3.9\cdot 10^{41}(\log{3.9\cdot 10^{41}})^3\) or \(n<2.8\cdot 10^{48}.\) 
Moreover, Lemma \ref{lbound} gives \(l<4.1\cdot 10^{13}\log n<4.1\cdot 10^{13}\log {(2.8\cdot 10^{48})}<4.6 \cdot  10^{15}\) and Lemma \ref{mbound} gives \(m<1.6\cdot 10^{27}(\log n)^2<m<1.6\cdot 10^{27}(\log {(2.8 \cdot  10^{48})})^2<2.0 \cdot  10^{31}\)
\end{proof}
The bounds established in Lemma \ref{nbound} cannot practically be computed and therefore require reduction. We now proceed with the reduction process in Subsection \ref{subsecl}.
\subsection{The reduction process}\label{subsecl}
Here, we apply Lemma \ref{duj} as follows. First, we return to the inequality \ref{eq:3.1} and put
\[
\Gamma = \frac{9}{d_1} \cdot \alpha^{n} \cdot 10^{-(2l+m)} - 1.
\]
Inequality \eqref{eq:3.1} can be rewritten as \(|\Gamma| < 28 \cdot 10^{-l}\).
If we assume that \(l \geq 2\), then \(28 \cdot 10^{-l} < 0.5\) holds. Assume for a moment \(l \geq 2\), then

\[
|\log{(\Gamma + 1)}| < 1.5 |\Gamma|,
\]
so that
\[
\left| \log\left(\frac{9}{d_1}\right) + n \log{\alpha} - (2l + m) \log{10} \right| < 42 \cdot 10^{-l}.\]
Dividing through by \(\log\alpha\) gives,
\begin{equation*}
\left|(2l+m)\frac{\log{10}}{\log{\alpha}} - n + \frac{\log{\left(\frac{9}{d_1}\right)}}{\log\alpha}\right| < \frac{42}{\log\alpha} \cdot 10^{-l}.
\end{equation*}
So we apply Lemma \ref{duj} with the data
\[
\tau := \frac{\log{10}}{\log{\alpha}}, \quad \mu(d_1) = \frac{\log\left(d_1/9\right)}{\log{\alpha}}, \quad d_1 \in \{1, 2, \ldots, 9\}, \quad A := \frac{42}{\log{\alpha}}, \quad B := 10, \quad k := l.
\]
Since \(2l + m - 2 < n\), then we can take \(M := 2.8 \cdot 10^{48}\). With the help of SageMath with the code in Appendix 2, we find that the convergent
\[
\frac{p}{q} = \frac{p_{87}}{q_{87}}=\frac{3265182491485655981489358337246995432669831208061478}{362926510191645833704423315164618426146198842188725}
\]
is such that \(q = q_{87} > 6M\). Furthermore, it gives \(\varepsilon > 0.4883316119\) and thus \(l < 54\). Therefore, we have that \(l < 54\). The case \(l < 2\) also holds because \(l < 2 < 54\).

Next, for fixed \(d_1, d_2 \in \{0, 1, 2, \ldots, 9\}\) with \(d_1 > 0\) and \(1 \leq l \leq 54\), we return to the inequality \eqref{eq:3.3} and put
\[
\Gamma_1 = \frac{9}{d_1 \cdot 10^l - (d_1 - d_2)} \cdot \alpha^n \cdot 10^{-l-m} - 1.
\]
From inequality \eqref{eq:3.3}, we have \(|\Gamma_1| < 19 \cdot 10^{-m}\). Assume that \(m \geq 2\), then \(\log (\Gamma_1 + 1) < 1.5 |\Gamma_1|\) holds. We therefore have
\[
\left|\log \left( \frac{9}{d_1 \cdot 10^l - (d_1 - d_2)} \right) + n \log \alpha - (l + m) \log 10 \right| < 30 \cdot 10^{-m}.
\]
Dividing through by \(\log \alpha\), we get
\[
\left|(l + m) \frac{\log 10}{\log \alpha} - n + \frac{\log \left( \frac{d_1 \cdot 10^l - (d_1 - d_2)}{9} \right)}{\log \alpha} \right| < \frac{30}{\log \alpha} \cdot 10^{-m}.
\]
Thus, we apply Lemma \ref{duj} with the quantities
\[
\mu(d_1, d_2) := \frac{\log \left( \frac{d_1 \cdot 10^l - (d_1 - d_2)}{9} \right)}{\log \alpha}, \quad d_1, d_2 \in \{0, 1, 2, \ldots, 9\} ~~\text{and}~~ d_1 > 0 \quad A := \frac{30}{\log \alpha}, \quad B := 10, \quad k = m.
\]
Since \(l + m < 2l + m\), we set \(M := 2.8 \cdot 10^{48}\) as an upper bound on \(l + m\). With the help of a simple computer program in SageMath (Appendix 3), we get that \(\varepsilon > 0.4994348950\), and therefore \(m \leq 57\). The case \(m < 2\) holds as well since \(m < 2 < 57\).

Lastly, for fixed \(d_1, d_2 \in \{0, 1, 2, \ldots, 9\}\) with \(d_1 > 0\), \(1 \leq m \leq 57\), and \(1 \leq l \leq 54\), we return to inequality \eqref{eq:3.5} and put
\[
\Gamma_2 = \frac{d_1 \cdot 10^{l+m} - (d_1 - d_2) \cdot 10^m + (d_1 - d_2)}{9} \cdot \alpha^{-n} \cdot 10^l - 1.
\]
From inequality \eqref{eq:3.5}, we have that \(|\Gamma_2| < (10/9) \cdot \alpha^{-n}\). Since \(n > 700\), the right--hand side of this inequality is less than \(1/2\), thus the above inequality implies that
\[
\left|l \log 10 - n \log \alpha + \log \left( \frac{d_1 \cdot 10^{l+m} - (d_1 - d_2) \cdot 10^m + (d_1 - d_2)}{9}\right) \right| < \frac{20}{9} \cdot \alpha^{-n}.
\]
Dividing through the above inequality by \(\log \alpha\) yields
\[
\left|l \frac{\log 10}{\log \alpha} - n + \frac{\log \left( \frac{d_1 \cdot 10^{l+m} - (d_1 - d_2) \cdot 10^m + (d_1 - d_2)}{9} \right)}{\log \alpha}\right| < \frac{20}{9 \log \alpha} \cdot \alpha^{-n}.
\]
Again, we apply Lemma \ref{duj} with the quantities
\[
\mu(d_1, d_2) := \frac{\log \left( \frac{d_1 \cdot 10^{l+m} - (d_1 - d_2) \cdot 10^m + (d_1 - d_2)}{9} \right)}{\log \alpha}, \quad A := \frac{20}{9 \log \alpha}, \quad B := \alpha, \quad k = n, \quad M := 2.8 \cdot 10^{48}.
\]
With the help of a simple computer program in SageMath (Appendix 4), we get \(\varepsilon >0.4995600863\) and thus \(n \leq 517\), contradicting our working assumption that \(n > 700\). Hence, Theorem \ref{thm1.1l} holds.

\section*{Conclusion}
In this study, we have demonstrated that 22 is the only Perrin number that can be expressed as a palindromic concatenation of two repdigits. This finding parallels previous research on Lucas numbers, which similarly identified no instances of Lucas numbers exhibiting such palindromic properties, see \cite{bat}. Further investigation into the concatenation of two $k$--generalized Perrin numbers, extending beyond the classical Perrin sequence, remains an open and intriguing area for future research. This exploration aims to uncover whether similar palindromic structures exist within the realm of generalized Perrin sequences, potentially shedding light on deeper connections and patterns within this family of sequences.

\section*{Acknowledgments} 
We thank the department of Mathematics at Makerere university for the conducive environment provided during this research period.

\section*{Address}
$ ^{1} $ Department of Mathematics, School of Physical Sciences, College of Natural Sciences, Makerere University, Kampala, Uganda

Email: \url{hbatte91@gmail.com}

Email: \url{kaggwaprosper58@gmail.com}
\pagebreak
\section*{Appendices}
\subsection*{Appendix 1}\label{app1}
\begin{verbatim}
def generate_perrin_sequence(n):
    Perrin = [3, 0, 2]  # Initial Perrin numbers
    while len(Perrin) < n:
        Perrin.append(Perrin[-2] + Perrin[-3])
    return Perrin  # Return the list of Perrin numbers

def is_palindromic(number):
    # Check if the number is palindromic
    return str(number) == str(number)[::-1]

def generate_palindromic_perrin_numbers():
    palindromic_perrin_numbers = set()
    
    perrin_numbers = generate_perrin_sequence(700)  # Generate first 700 Perrin numbers
    
    for perrin_number in perrin_numbers:
        perrin_str = str(perrin_number)
        
        # Check if the number is even length and palindromic
        if len(perrin_str) % 2 == 0 and is_palindromic(perrin_number):
            half_length = len(perrin_str) // 2
            first_half = perrin_str[:half_length]
            palindromic_number = int(first_half + first_half[::-1])
            
            # Check if the palindromic number is a Perrin number
            if palindromic_number in perrin_numbers:
                palindromic_perrin_numbers.add(palindromic_number)
    
    return palindromic_perrin_numbers

# Generate and print palindromic Perrin numbers
palindromic_perrin_numbers = generate_palindromic_perrin_numbers()
print("Palindromic Perrin numbers:")
print(sorted(palindromic_perrin_numbers))

\end{verbatim}

\subsection*{Appendix 2}\label{app2}
\begin{verbatim}
from sage.all import *

# Constants
r1 = (31 + sqrt(69)) / 8
r2 = (31 - sqrt(69)) / 8
a = (r1 + r2) / 6  # This is the value of alpha
a = a.n(digits=1000)  # Ensure the precision is 1000 digits
tau = (log(10) / log(a)).n(digits=1000)
A = (42 / log(a))
B = 10
M = 2.8 * 10^48

# Continued Fraction and Convergents
cf_tau = continued_fraction(tau)
convergents = cf_tau.convergents()

for d1 in range(1, 10):  # Iterate through d1 from 1 to 9
    mu = (log(d1/9) / log(a)).n(digits=1000)

    DD = []  # Initialize empty list for results for each d1

    # Dujella and Petho Reduction Method
    for i, convergent in enumerate(convergents):
        p, q = convergent.numerator(), convergent.denominator()
        ep = abs(mu * q - round(mu * q)) - M * abs(tau * q - round(tau * q))

        if q > 6 * M and ep > 0:
            log_expr_a = (log(A * q / ep) / log(B)).n(digits=10)
            DD.append((i, ep.n(digits=10), log_expr_a))
            print(f"d1 = {d1}, p_{i}/q_{i} = {p}/{q}")
            break  # Stop after finding the first suitable convergent for this d1

    # Results for each d1
    if DD:
        print(f"Results for d1 = {d1}:")
        print("First few elements of DD:", DD[:1])
    else:
        print(f"No suitable convergent found for d1 = {d1}.")
print("Continued fraction expansion of tau:", cf_tau[:20])
\end{verbatim}

\subsection*{Appendix 3}\label{app3}
\begin{verbatim}
from sage.all import *

# Constants
r1 = (31 + sqrt(69)) / 8
r2 = (31 - sqrt(69)) / 8
a = (r1 + r2) / 6  # This is the value of alpha
a = a.n(digits=1000)  # Ensure the precision is 1000 digits
tau = (log(10) / log(a)).n(digits=1000)
A = (30 / log(a))
B = 10
M = 2.8 * 10^48

# Continued Fraction and Convergents
cf_tau = continued_fraction(tau)
convergents = cf_tau.convergents()

# Variables to store maximum values
max_ep = -Infinity
max_log_expr_a = -Infinity

for d1 in range(1, 10):
    for d2 in range(0, 10):
        
        for l in range(1, 54):
           mu = (log((d1 * 10^l - (d1 - d2)) / 9) / log(a)).n(digits=1000)

           # Dujella and Petho Reduction Method
           for i, convergent in enumerate(convergents):
               p, q = convergent.numerator(), convergent.denominator()
               ep = abs(mu * q - round(mu * q)) - M * abs(tau * q - round(tau * q))

               if q > 6 * M and ep > 0:
                  log_expr_a = (log(A * q / ep) / log(B)).n(digits=10)
                  if ep > max_ep:
                     max_ep = ep.n(digits=10)
                  if log_expr_a > max_log_expr_a:
                     max_log_expr_a = log_expr_a
                  break  # Stop after finding the first suitable convergent

# Print maximum values
print("Maximum ep across all combinations:", max_ep)
print("Maximum log_expr_a across all combinations:", max_log_expr_a)
\end{verbatim}

\subsection*{Appendix 4}\label{app4}
\begin{verbatim}
from sage.all import *
# Constants
r1 = (31 + sqrt(69)) / 8
r2 = (31 - sqrt(69)) / 8
a = (r1 + r2) / 6  # This is the value of alpha
a = a.n(digits=1000)  # Ensure the precision is 1000 digits
tau = (log(10) / log(a)).n(digits=1000)
A = (20 / (9 * log(a)))
B = a
M = 2.8 * 10^48
# Continued Fraction and Convergents
cf_tau = continued_fraction(tau)
convergents = cf_tau.convergents()
# Variables to store maximum values
max_ep = -Infinity
max_log_expr_a = -Infinity
# Iterate through combinations of d1, d2, l, and m
for d1 in range(1, 10):
    for d2 in range(0, 10):
      
      for l in range(1, 54):
         for m in range(1, 57):
            mu = (log((d1 * 10^(l+m) - (d1 - d2) * 10^m + 
            (d1 - d2)) / 9) / log(a)).n(digits=1000)
            # Dujella and Petho Reduction Method
            for convergent in convergents:
                p, q = convergent.numerator(), convergent.denominator()
                ep = abs(mu * q - round(mu * q)) - M * abs(tau * q - round(tau * q))
                if q > 6 * M and ep > 0:
                   log_expr_a = (log(A * q / ep) / log(B)).n(digits=10)
                   if ep > max_ep:
                      max_ep = ep.n(digits=10)
                   if log_expr_a > max_log_expr_a:
                      max_log_expr_a = log_expr_a
                   break  # Stop after finding the first suitable convergent
# Print maximum values
print("Maximum ep across all combinations:", max_ep)
print("Maximum log_expr_a across all combinations:", max_log_expr_a)    

\end{verbatim}	
\end{document}